\documentclass[12pt]{amsart}
\usepackage{amsmath, color}
\usepackage{amssymb}
\usepackage{amsfonts}
\usepackage{amsthm}
\usepackage{mathrsfs}
\usepackage[all]{xy}

\newcommand{\be}{\begin{equation}}
\newcommand{\ee}{\end{equation}}

\theoremstyle{definition}

\theoremstyle{remark}

\theoremstyle{conjecture}

\newtheorem{lem}{Lemma}
\newtheorem{thm}{Theorem}

\newtheorem{pro}{Proposition}

\numberwithin{equation}{section}

\begin{document}

\title{Twisted quantum toroidal algebras $T_q^-(\mathfrak g)$}
\author{Naihuan Jing, Rongjia Liu$^*$}
\address{Department of Mathematics, North Carolina State University, Raleigh, NC 27695, USA
and School of Sciences, South China University of Technology,
Guangzhou 510640, China}
\email{jing@math.ncsu.edu}
\address{School of Sciences, South China University of Technology,
Guangzhou 510640, China}
\email{liu.rongjia@mail.scut.edu.cn}

\thanks{{\scriptsize
\hskip -0.4 true cm MSC (2010): Primary: 17B65; Secondary: 17B67, 17B69.
$*$Corresponding author.
}}

\maketitle

\begin{abstract} We construct a principally graded quantum loop algebra
for the Kac-Moody algebra. As a special case a twisted analog of the quantum toroidal algebra
is obtained together with the quantum Serre relations.

\end{abstract}

\section{Introduction}
Toroidal algebras are an important class of algebras in the theory of the extended affine Lie algebras. The first realization of a toroidal algebra appeared as
a vertex representation of the affinized Kac-Moody algebra \cite{F}. The
loop algebraic presentation of toroidal Lie algebras was given by Moody et al.~\cite{MRY},
which shows the similarity with the affine Kac-Moody algebras.
Quantum toroidal algebras appeared in the study of the Langlands reciprocity for algebraic surfaces \cite{GKV}
and was constructed geometrically by Nakajima's quiver varieties \cite{N}. More general quantum toroidal
 Kac-Moody algebras were constructed in \cite{J2} using Drinfeld presentation and vertex representations
 where the quantum Serre relations were found to be closely connected with some nontrivial relations
 of Hall-Littlewood symmetric functions.
The quantum toroidal algebras have been studied
in various contexts: toroidal Schur-Weyl duality
 \cite{VV}, its general representation theory \cite{M},
 vertex representations \cite{S}, McKay correspondence
 \cite{FJW}, toroidal actions on level one
representations \cite{STU}, higher level analogs for quantum affine algebras \cite{TU},
fusion products \cite{H1} and an excellent survey can be found in \cite{H2}.
Recently quantum toroidal algebras have found more interesting rich structures and applications
in \cite{FJM1, FJM2, FJM3}.

Like the theory of affine algebras, quantum toroidal algebras also have twisted analogs. Based on the quantum
general linear algebra action on
the quantum toroidal algebra \cite{GJ}, a quantum Tits-Kantor-Kocher (TKK) algebra was constructed \cite{GJ2} using homogeneous $q$-deformed vertex operators in connection with special unitary Lie
algebras, where the Serre relation was also found to be equivalent to
some combinatorial identity of Hall-Littlewood polynomials.
More recently, a new twisted quantum toroidal algebra of type $A_1$ has also been constructed
as an analog of the quantum TKK algebra \cite{JL}.
On the other hand, a principal quantum affine algebra was constructed \cite{J3}
by deforming the Kac-Moody algebras with an involution. However that algebra is larger than its classical analog
as the Serre relations were not known. In the original classical situation the principal realization of affine Lie
algebras \cite{LW} has played an important role in the theory of vertex operator algebra \cite{FLM}.
In light of quantum homogeneous constructions \cite{FJ, J1}, the quantum principal realization was also expected.
In this paper, we construct a principally graded quantum Kac-Moody algebras and obtain
the associated Serre relations, which will guarantee the finiteness of the representation theory.

The content is organized as follows. Section 2 prepares some background information  regarding twisted quantum toroidal algebras. Section three introduces a twisted quantum Heisenberg algebra associated to the root lattice of the Kac-Moody algebra and then its Fock space representation is constructed. After that, Section four is concerned with proving
the quantum Serre relations of $(-1)$-twisted quantum toroidal algebras.
 \\[0.1cm]

\section{$(-1)$-Twisted quantum toroidal algebras}
We start by recalling the construction
of the principally graded quantum algebras constructed in \cite{J2} and
then state the newly found Serre relations. Let $\mathfrak g$ be the complex finite dimensional simple Lie algebra of a simply laced type. Let $\alpha_{1}, \cdots,\alpha_{l}$ be the simple roots. The normalized invariant bilinear form $(\cdot|\cdot)$ on $\mathfrak g(A)$ satisfies the property that $(\alpha_{i}|\alpha_{j})=a_{ij}$, where $A=(a_{ij})$
is the associated Cartan matrix of $\mathfrak g$.

The twisted quantum toroidal algebra $T_q^-({\mathfrak g})$ is the complex unital associative algebra generated by
\begin{equation*}
 q^{c},\quad h_{im},
\quad x_{in}^{\pm}, ~m\in2\mathbb{Z}+1, ~n\in \mathbb{Z},
\quad i=0,\ldots, l,
\end{equation*}
subject to the relations written in terms of generating series.

Let
\begin{equation}
 x_i^{\pm}(z)=\sum_{n\in \mathbb{Z}}x_{in}^{\pm} z^{-n},
\end{equation}
\begin{equation}
\phi_{i}(z)=\exp\big\{(q^{-1}-q) 2\sum_{m\in2\mathbb{N}-1}^{\infty} h_{i, -m}z^{m}\big\}=\sum_{n\geqslant 0}\phi_{i,-n}z^{n},
\end{equation}
\begin{equation}
\psi_{i}(z)=\exp\big\{(q- q^{-1}) 2\sum_{m\in2\mathbb{N}-1}^{\infty}h_{i, m}z^{-m}\big\}=\sum_{n\geqslant 0}\psi_{i,n}z^{-n},
\end{equation}
where $\mathbb{N}$ is the set of natural numbers $1, 2, \cdots$. Then the relations of the twisted toroidal algebra are the following:
\begin{equation}
[q^c, \phi_{i}(z)]=[q^c, \psi_{i}(z)]=[q^c, x_{i}^{\pm}(z)]=0,
\end{equation}
\begin{equation}
[x_{i}^{\pm}(z), x_{j}^{\pm}(w)]= [x_{i}^{+}(z), x_{j}^{-}(w)]=0, ~~{if}~(\alpha_{i}|\alpha_{j})=0,
\end{equation}
\begin{equation}
(z+w)[x_{i}^{+}(z), x_{j}^{-}(w)]=0, ~~{if}~(\alpha_{i}|\alpha_{j})=-1,
\end{equation}
\begin{equation}
[x_{i}^{+}(z), x_{i}^{-}(w)]= \frac{2(q+q^{-1})}{q-q^{-1}}\Big(\psi_{i}(q^{-c/2}z)\delta(\frac{w}{z}q^{c})
-\phi_{i}(q^{c/2}z)\delta(\frac{w}{z}q^{-c})\Big).
\end{equation}
The Serre relations are given as follows:
 \begin{eqnarray}
\begin{split}
Sym_{z_{1},z_{2}}\Big\{(z_{1}+q^{\mp2}z_{2})(z_{2}-q^{\mp2}z_{1})\Big(z_{2}x_i ^{\pm}(z_1 )x_i ^{\pm}(z_2 )x_j^{\pm} (w)-(z_{1}+z_{2})\\\cdot x_i ^{\pm}(z_1 )x_j^{\pm} (w)x_i ^{\pm}(z_2 )+z_{1}x_j^{\pm} (w)x_i ^{\pm}(z_1 )x_i ^{\pm}(z_2 )\Big)\Big\}=0, \quad if ~(\alpha_{i}|\alpha_{j})= -1,
\end{split}
\end{eqnarray}
\begin{eqnarray}
\begin{split}
\sum_{r=0,\sigma \in \mathfrak{S}_{k+1}}^{k+1}
\sigma .\Big(\prod_{m<n}(z_{m}+q^{\mp2}z_{n})(z_{n}-q^{\mp2}z_{m})x_i ^{\pm} (z_1 )x_i ^{\pm}(z_2 )
\cdots x_i ^{\pm}(z_{r})\\ \cdot x_j^{\pm} (w)x_i ^{\pm}(z_{r+1})\cdots x_i ^{\pm}(z_{k+1} )\Big)=0,~if~(\alpha_{i}|\alpha_{j})= -k, ~k\in \mathbb N,~k\geqslant2.
\end{split}
\end{eqnarray}
 Let
 \begin{equation*}
 \displaystyle G_{ij}(x)=\sum_{n=0}^{\infty}G_{n}x^n
 \end{equation*}
 be the Taylor series at $x=0$ of the following functions
 \begin{align}
 G_{ij}(x) &=\frac{q^{(\alpha_{i}|\alpha_{j})}x-1}{q^{(\alpha_{i}|\alpha_{j})}x+1}\frac{x+q^{(\alpha_{i}|\alpha_{j})}}{x-q^{(\alpha_{i}|\alpha_{j})}},\nonumber
\end{align}
 then the relations  of the twisted quantum {toroidal} algebra are expressed in terms of generating series:
\begin{eqnarray}
\phi_{i}(z)\psi_{j}(w)=\psi_{j}(w)\phi_{i}(z)G_{ij}(q^{-c}\frac{z}{w})/G_{ij}(q^{c}\frac{z}{w}),
\end{eqnarray}
\begin{equation}
[\phi_{i}(z),\phi_{j}(w)]=[\psi_{i}(z),\psi_{j}(w)]=0,
\end{equation}
\begin{align}\label{commutator1}
\phi_i(z)x_j^{\pm}(w)\phi_i(z)^{-1}&=x_j^{\pm}(w)G_{ij}(\frac{z}{w}q^{\mp c/2})^{\pm 1},\\ \label{commutator2}
\psi_i(z)x_j^{\pm}(w)\psi_i(z)^{-1}&=x_j^{\pm}(w)G_{ij}(\frac{w}{z}q^{\mp c/2})^{\pm 1},
\end{align}
where $\delta(z)=\sum_{n\in \mathbb{Z}}z^{n}$ is the formal $\delta$-function.\\[-0.1cm]

\section{Fock space representations}

The twisted quantum Heisenberg algebra $U_{q}(\widetilde{h})$ is the associative algebra generated by $a_{i}(m)$ $(m\in  2\mathbb{Z}+1)$ and the central element $\gamma=q^c$ subject to the following relations:
\begin{equation}
[a_{i}(m),a_{j}(n)]=\delta_{m,-n}\frac{[(\alpha_{i}|\alpha_{j})m]}{2m}\frac{\gamma^{m}-\gamma^{-m}}{q-q^{-1}},
\end{equation}
\begin{equation}
[a_{i}(m),\gamma]=[a_{i}(m),\gamma^{-1}]=0.
\end{equation}

Let $Q$ be the root lattice of the simply finite dimensional Lie algebra $\mathfrak g$ of the simply laced type with the standard bilinear form given by $(\alpha_{i}|\alpha_{j})=a_{ij}$.  Then let $\hat{Q}$ be the central extension of the root lattice $Q$ such that
\begin{equation}
1\longrightarrow \mathbb{Z}_{2} \longrightarrow \hat{Q}\longrightarrow Q\longrightarrow 1
\end{equation}
with the commutator
\begin{equation}
aba^{-1}b^{-1}=(-1)^{(\alpha|\beta)},
\end{equation}
where $a$ and $b$ are the preimages of $\alpha$ and $\beta$ respectively.

Let
\begin{equation}
S(\mathcal{H^{-}})=\mathbb{C}\left[a_{i}(n):1 \leqslant i \leqslant {l}, n\in -(2\mathbb{N}-1)\right]
\end{equation}
be the symmetric algebra generated by $a_{i}(n),  1 \leqslant i \leqslant {l},  n \in -(2\mathbb{N}-1)$.
Then $S(\mathcal{H^{-}})$ is an
$\mathcal{H}$-module under the action that
$a_{i}(n)$ acts as a differential operator for $n \in 2\mathbb{N}-1$,
and $a_{i}(n)$ acts as a multiplication operator for $n \in -(2\mathbb{N}-1)$.

Denote the preimage of $\alpha_{i}$ by $a_{i}$, and let $T$ be the $\hat{Q}$-module such that
\begin{equation}
a_{i}a_{j}=(-1)^{(\alpha_{i}|\alpha_{j})}a_{j}a_{i},
\end{equation}
the Fock space is defined as
\begin{equation}
V_{Q}=S(\mathcal{H}^{-})\otimes T.
\end{equation}

Introduce the twisted vertex operators acting on $V_{Q}$ as follows:
\begin{equation}
E_{-}^{\pm}(\alpha_{i},z)=\exp\Big(\pm\sum_{n=1, odd}^{\infty}\frac{2q^{\mp n/2}}{[n]}a_{i}(-n)z^{n}\Big),
\end{equation}
\begin{equation}
E_{+}^{\pm}(\alpha_{i},z)=\exp\Big(\mp\sum_{n=1, odd}^{\infty}\frac{2q^{\mp n/2}}{[n]}a_{i}(n)z^{-n}\Big),
\end{equation}
\begin{equation}
X_{i}^{\pm}(z)=E_{-}^{\pm}(\alpha_{i},z)E_{+}^{\pm}(\alpha_{i},z)a_{i}^{\pm1}=\sum_{n\in \mathbb{Z}}X_{i}^{\pm}(n)z^{-n}.
\end{equation}
Then we have the following result.
\begin{thm}
The space $F$ is a level one module for
the twisted quantum toroidal algebra $T_q^-({\mathfrak g})$ under the action defined by $\gamma\mapsto q$, $h_{im}~\mapsto$ $a_{i}(m)$, and $x^{\pm}_{i,n}$ $\mapsto$ $X_{i}^{\pm}(n)$.
\end{thm}

We will prove the theorem in this and the next section. To compute the operator product expansion we need the following $q$-analogs of series $(z-w)^{r}$ introduced in \cite{J1, J2}. For $r\in\mathbb{C}$, we call
the following $q$-analogs:
\begin{equation}
(a; q)_{\infty}=\Pi_{n=0}^{\infty}(1-aq^{n}),
\end{equation}
\begin{equation}
(1-z)_{q^{2}}^{r}=\frac{(q^{-r+1}z; q^{2})_{\infty}}{(q^{r+1}z; q^{2})_{\infty}}=exp(-\sum_{n=1}^{\infty}\frac{[rn]}{n[n]}z^{n}).
\end{equation}
Similarly the twisted q-analog is defined by \cite{J3}:
\begin{equation}
(\frac{1-z}{1+z})_{q^{2}}^{r}=\frac{(1-z)_{q^{2}}^{r}}{(1+z)_{q^{2}}^{r}}=exp(-\sum_{n\in 2\mathbb{N}-1}\frac{2[rn]}{n[n]}z^{n}).
\end{equation}

The operator product expansions (OPE) for $X_{i}^{\pm}$(z) are given by
\begin{eqnarray}\label{A}
X_{i}^{\pm}(z)X_{j}^{\pm}(w)=\!\!\!\!&&\!\!\!\!:X_{i}^{\pm}(z)X_{j}^{\pm}(w):
\left(\frac{1-q^{\mp1}w/z}{1+q^{\mp1}w/z}\right)_{q^{2}}^{(\alpha_{i}|\alpha_{j})},
\\X_{i}^{\pm}(z)X_{j}^{\mp}(w)=\!\!\!\!&&\!\!\!\!:X_{i}^{\pm}(z)X_{j}^{\mp}(w):
\left(\frac{1+w/z}{1-w/z}\right)_{q^{2}}^{(\alpha_{i}|\alpha_{j})}.
\end{eqnarray}

In particular, when $(\alpha_{i}|\alpha_{j})=-1$,
\begin{eqnarray}
X_{i}^{\pm}(z)X_{j}^{\pm}(w)=\!\!\!\!&&\!\!\!\!:X_{i}^{\pm}(z)X_{j}^{\pm}(w):
\frac{z+q^{\mp1}w}{z-q^{\mp1}w},
\\X_{i}^{\pm}(z)X_{j}^{\mp}(w)=\!\!\!\!&&\!\!\!\!:X_{i}^{\pm}(z)X_{j}^{\mp}(w):
\frac{z-w}{z+w} .
\end{eqnarray}

\begin{lem}\cite{J3} The operators ~$\phi_{i}(zq^{-1/2})=:X_{i}^{+}(zq^{-1})X_{i}^{-}(z):$  ~and ~$\psi_{i}(zq^{1/2})=\\:X_{i}^{+}(zq)X_{i}^{-}(z):$  are given by
\begin{equation}
\phi_{i}(z)=\exp\big\{(q^{-1}-q) 2\sum_{m\in 2\mathbb{N}-1} h_{i, -m}z^{m}\big\}=\sum_{n\geqslant 0}\phi_{i,-n}z^{n},
\end{equation}
\begin{equation}
\psi_{i}(z)=\exp\big\{(q- q^{-1}) 2\sum_{m\in 2\mathbb{N}-1}h_{i, m}z^{-m}\big\}=\sum_{n\geqslant 0}\psi_{i,n}z^{-n}.
\end{equation}
\end{lem}

\section{Twisted quantum Serre relations}

Now we use the vertex representation and quantum vertex operators to prove that the Serre relations are
satisfied by our representation. In the following we treat the ``$+$''-case, as the ``$-$''-case can be proved similarly.

\begin{pro}\label{21}If $(\alpha_{i}|\alpha_{j})= -1$, the ``$+$''-Serre relation can be written as:
\begin{eqnarray}\label{marker}
\begin{split}
Sym_{z_{1},z_{2}}\Big\{(z_{1}+q^{-2}z_{2})(z_{2}-q^{-2}z_{1})\Big(z_{2}X_i^{+} (z_1 )X_i^{+}(z_2 )X_j^{+}(w)\\-(z_{1}+z_{2})X_i^{+}(z_1 )X_j^{+}(w)X_i^{+}(z_2 )+z_{1}X_j^{+}(w)X_i^{+}(z_1 )X_i^{+}(z_2 )\Big)\Big\}=0.
\end{split}
\end{eqnarray}
\end{pro}

\begin{proof}
When $(\alpha_{i}|\alpha_{j})= -1$, the OPEs are
\begin{align*}
X_{i}^{+}(z_{1})X_{j}^{+}(w)&=:X_{i}^{+}(z_{1})X_{j}^{+}(w):\frac{z+q^{-1}w}{z-q^{-1}w},\\
X_{i}^{+}(z_{1})X_{i}^{+}(z_{2})&=:X_{i}^{+}(z_{1})X_{i}^{+}(z_{2}):\frac{z_{1}-z_{2}}{z_{1}+z_{2}}
\cdot\frac{z_{1}-q^{-2}z_{2}}{z_{1}+q^{-2}z_{2}}.
\end{align*}
Thus the bracket inside the left-hand side of Eq. (\ref{marker}) is simplified as
\begin{eqnarray*}
(z_{1}+\!\!\!\!&&\!\!\!\!\!\!q^{-2}z_{2})(z_{2}-q^{-2}z_{1})\Big(z_{2}X_i^{+} (z_1 )X_i^{+} (z_2 )X_j^{+} (w)
\\-(z_{1}+\!\!\!\!\!\!&&\!\!\!\!z_{2})X_i^{+} (z_1 )X_j^{+} (w)X_i^{+} (z_2 )+z_{1}X_j^{+} (w)X_i^{+} (z_1 )X_i^{+}(z_2 )\Big)
\\=\!\!\!\!&&\!\!\!\!:X_i^{+} (z_1 )X_i^{+} (z_2 )X_j^{+} (w):\frac{(z_{1}-z_{2})(z_{1}-q^{-2}z_{2})(z_{2}-q^{-2}z_{1})}{z_{1}+z_{2}}
\\\!\!\!\!&&\!\!\!\!\Big( z_{2}\cdot\frac{z_{1}+q^{-1}w}{z_{1}-q^{-1}w}\cdot\frac{z_{2}+q^{-1}w}{z_{2}-q^{-1}w}
+(z_{1}+z_{2})
\\\!\!\!\!&&\!\!\!\!\cdot \frac{z_{1}+q^{-1}w}{z_{1}-q^{-1}w}\cdot\frac{w+q^{-1}z_{2}}{w-q^{-1}z_{2}}+z_{1}\frac{w+q^{-1}z_{1}}{w-q^{-1}z_{1}}\cdot\frac{w+q^{-1}z_{2}}{w-q^{-1}z_{2}}\Big)
\\=\!\!\!\!&&\!\!\!\!:X_i^{+} (z_1 )X_i^{+} (z_2 )X_j^{+} (w):\frac{z_{1}-z_{2}}{z_{1}+z_{2}}\prod_{i=1}^{2}\frac{(w-q^{-1}z_{i})^{-1}}{(z_{i}-q^{-1}w)}
\\\!\!\!\!&&\!\!\!\!\cdot \Big\{z_{2}(z_{1}+q^{-1}w)(z_{2}+q^{-1}w)(w-q^{-1}z_{1})(w-q^{-1}z_{2})
\\\!\!\!\!&&\!\!\!\!+(z_{1}+z_{2})(z_{1}+q^{-1}w)(z_{2}-q^{-1}w)(w-q^{-1}z_{1})(w+q^{-1}z_{2})
\\\!\!\!\!&&\!\!\!\!+z_{1}(z_{1}-q^{-1}w)(z_{2}-q^{-1}w)(w+q^{-1}z_{1})(w+q^{-1}z_{2})\Big\}.
\end{eqnarray*}
The proposition is proved if the following lemma holds. \end{proof}
\begin{lem}
Let $\mathfrak{S}_{2}$ act on $z_{1},z_{2}~via~\sigma.z_{i}=z_{\sigma(i)}$. Then
\begin{eqnarray*}
\sum_{\sigma\in\mathfrak{S}_{2}}\sigma.\Bigg\{(z_{1}-z_{2})\Big(z_{2}(z_{1}+q^{-1}w)(z_{2}+q^{-1}w)(w-q^{-1}z_{1})(w-q^{-1}z_{2})
\\+(z_{1}+z_{2})(z_{1}+q^{-1}w)(z_{2}-q^{-1}w)(w-q^{-1}z_{1})(w+q^{-1}z_{2})
\\+z_{1}(z_{1}-q^{-1}w)(z_{2}-q^{-1}w)(w+q^{-1}z_{1})(w+q^{-1}z_{2})\Big)\Bigg\}=0.
\end{eqnarray*}
\end{lem}
\begin{proof}
Considering the left-hand side as a polynomial in $w$, we extract the constant term:
\begin{equation*}
\sum \limits_{\sigma  \in \mathfrak{S}_2 }\sigma.(z_{1}-z_{2})\Big(z_{2}q^{-2}z_{1}^{2}z_{2}^{2}+
(z_{1}+z_{2})(-q^{-2}z_{1}^{2}z_{2}^{2})+z_{1}q^{-2}z_{1}^{2}z_{2}^{2}\Big)=0.
\end{equation*}
Similarly the highest coefficient of $w^{4}$ is seen to be zero.
The coefficients of $w$ and $w^3$ are respectively computed as follows.
\begin{align*}
\sum \limits_{\sigma  \in \mathfrak{S}_3 }&\sigma.
(z_{1}-z_{2})\big(z_{2}(z_{1}+z_{2})+(z_{1}+z_{2})(z_{1}-z_{2})-z_{1}(z_{1}+z_{2})
\big)=0,\\
\sum \limits_{\sigma  \in \mathfrak{S}_3 }&\sigma.
(z_{1}-z_{2})\big[(z_{1}+z_{2})((q^{2}+q^{-2})z_{1}z_{2}-(z_{1}+z_{2})^{2})
\\&\qquad+(z_{1}+z_{2})((q^{2}+q^{-2})z_{1}z_{2}+(z_{1}-z_{2})^{2})
\big]=0.
\end{align*}
Thus Lemma 1 is proved.
\end{proof}


Until now, we have discussed the Serre relation
with $(\alpha_{i}|\alpha_{j})= -1$. Next we consider the general situation when $(\alpha_{i}|\alpha_{j})= -k,~k\in \mathbb{N}, ~k\geqslant2$. The twisted quantum Serre relations are proved generally
as follows.

\begin{pro}\label{32} If $(\alpha_{i}|\alpha_{j})= -k, ~k\in \mathbb{N},~k\geqslant2$, then
\begin{eqnarray}\label{B}
\begin{split}
\sum_{r=0,\sigma \in \mathfrak{S}_{k+1}}^{k+1}\sigma .\Big(\prod_{m<n}(z_{m}+q^{-2}z_{n})(z_{n}-q^{-2}z_{m})X_i^{+} (z_1 )X_i^{+}(z_2 )
\cdots \\ \cdot X_{i}^{+}(z_{r}) X_j^{+} (w)X_i^{+} (z_{r+1})\cdots X_i^{+} (z_{k+1} )\Big)=0,
\end{split}
\end{eqnarray}
where the symmetric group $\mathfrak{S}_{k+1}$ acts on the variables $z_1, \cdots, z_{k+1}$.
\end{pro}
\begin{proof} For $(\alpha_{i}|\alpha_{j})= -k, ~k\in \mathbb{N},~k\geqslant2$, from (\ref{A}) it follows that
\begin{equation*}
X_{i}^{+}(z_{1})X_{i}^{+}(z_{2})=:X_{i}^{+}(z_{1})X_{i}^{+}(z_{2}):
\frac{z_{1}-z_{2}}{z_{1}+z_{2}}\cdot\frac{z_{1}-q^{-2}z_{2}}{z_{1}+q^{-2}z_{2}},
\end{equation*}
\begin{equation*}
X_{i}^{+}(z)X_{j}^{+}(w)=:X_{i}^{+}(z)X_{j}^{+}(w):\frac{z+q^{-k}w}{z-q^{-k}w}
\cdot\frac{z+q^{-k+2}w}{z-q^{-k+2}w}\cdots\frac{z+q^{k-2}w}{z-q^{k-2}w}.
\end{equation*}
Let
\begin{equation*}
[z, w; k]_{q^{2}}=(z-w)(z-wq^{2})\cdots(z-wq^{2(k-1)}).
\end{equation*}
Then the left-hand side of (\ref{B}) can be simplified as
\begin{align*}
&\sum_{r=0,\sigma \in \mathfrak{S}_{k+1}}^{k+1}\sigma .\Big(\prod_{m<n}(z_{m}+q^{-2}z_{n})
(z_{n}-q^{-2}z_{m})X_i^{+} (z_1 )\cdots X_{i}^{+}(z_{r})\\
&\qquad\quad X_j^{+} (w)\cdot X_i^{+} (z_{r+1})\cdots X_i^{+} (z_{k+1} )\Big)
\\=&:X_j^{+} (w )\prod_{r=1}^{k+1}X_i^{+} (z_r):
\sum_{\sigma \in \mathfrak{S}_{k+1}}\sigma .\big(\frac{\prod_{m<n}(z_{m}-z_{n})}{\prod_{l=1}^{k+1}(w-q^{-1}z_{l})(z_{l}-q^{-1}w)}\cdot f(z_{1},z_{2},\cdots,z_{k+1})\big)\\
=&:X_j^{+} (w )\prod_{r=1}^{k+1}X_i^{+} (z_r):\frac{\prod_{m<n}(z_{m}-z_{n})}{\prod_{l=1}^{k+1}(w-q^{-1}z_{l})(z_{l}-q^{-1}w)}
\sum_{\sigma \in \mathfrak{S}_{k+1}}sgn(\sigma)\sigma.f(z_{1},z_{2},\cdots,z_{k+1}),
\end{align*}
where we have put (for $k$ even)
\begin{align*}
f&(z_{1},z_{2},\cdots,z_{k+1})\\
&=[w,-z_{1}q^{-k}]_{q^{2}}\cdots[w,-z_{k+1}q^{-k}]_{q^{2}}[z_{1},wq^{-k}]_{q^{2}}\cdots[z_{k+1},wq^{-k}]_{q^{2}}
\\&\quad+[z_{1},-wq^{-k}]_{q^{2}}[w,-z_{2}q^{-k}]_{q^{2}}
\cdots[w,-z_{k+1}q^{-k}]_{q^{2}}\\
&\qquad\qquad\cdot[w,z_{1}q^{-k}]_{q^{2}}[z_{2},wq^{-k}]_{q^{2}}\cdots[z_{k+1},wq^{-k}]_{q^{2}}+\cdots\\
&\quad+[z_{1},-wq^{-k}]_{q^{2}}\cdots[z_{k+1},-wq^{-k}]_{q^{2}}[w,z_{1}q^{-k}]_{q^{2}}\cdots[w,z_{k+1}q^{-k}]_{q^{2}}.
\end{align*}
Here we briefly write
$[z, w; k]_{q^2}$ as $[z, w]_{q^2}$.
Observe that each summand in $f(z_1, \cdots, z_{k+1})$ has at least one symmetry under a transposition.
For example
$$[z_{1},-wq^{-k}]_{q^{2}}[w,-z_{2}q^{-k}]_{q^{2}}\cdots[w,-z_{k+1}q^{-k}]_{q^{2}}
[w,z_{1}q^{-k}]_{q^{2}}[z_{2},wq^{-k}]_{q^{2}}\cdots[z_{k+1},wq^{-k}]_{q^{2}}$$
 is invariant under
switching $z_2$ by $z_3$ when $k\geq 2$. Therefore the antisymmetrizer of this
summand under $\mathfrak S_{k+1}$ is zero. Subsequently $\sum_{\sigma \in \mathfrak{S}_{k+1}}sgn(\sigma)\sigma .f(z_{1},z_{2},\cdots,z_{k+1})=0$, and the Serre relation is proved.
\end{proof}

\bigskip

\centerline{\bf Acknowledgments}

 NJ gratefully acknowledges the support of
Humboldt Foundation, Simons Foundation
grant 198129, and NSFC grant 11271138 during this work.

\bigskip

\bibliographystyle{amsalpha}

\end{document}